\documentclass[
11pt,%
tightenlines,%
twoside,%
onecolumn,%
nofloats,%
nobibnotes,%
nofootinbib,%
superscriptaddress,%
noshowpacs,%
centertags]%
{revtex4-2}
\usepackage{ljm}

\begin{document}

	\titlerunning{Chambers and walls} 
	\authorrunning{V.I. Zvonilov} 
		
\title{Chambers and walls in spaces of real algebraic curves of small degrees}
		
	\author{\firstname{V.~I.}~\surname{Zvonilov}}
	\email[E-mail: ]{zvonilov@itmm.unn.ru}
	\affiliation{Institute of Information Technology, Mathematics and Mechanics, Lobachevsky State University of Nizhny Novgorod, 23 Prospekt Gagarina, Nizhny Novgorod, 603950 Russia}
		
	\received{September 30, 2025} 

	\begin{abstract} 
		This paper reviews known results on the rigid isotopy classification of plane curves of degree $m\leq6$ and curves of small degrees on quadrics. The paper's study completes the rigid isotopy classification of nonsingular real algebraic curves of		bidegree (4,3) on a hyperboloid, begun by the author in earlier works. There are given previously missing proofs of the uniqueness of the connected components for 16 classes of real algebraic curves of bidegree (4,3)
		having a single node or a cusp.  The main technical tools are graphs of real trigonal curves on Hirzebruch surfaces.  Adjacency graphs of chambers and walls in the spaces of these curves are presented.
	\end{abstract}
	
	
	
	\maketitle
\newtheorem{remark}{Remark}
\newtheorem{proposition}{Proposition}	
	
\newcommand{\R}{\mathbb{R}}
\newcommand{\Cb}{\mathbb{C}}
\newcommand{\Z}{\mathbb{Z}}
\newcommand{\N}{\mathbb{N}}
\newcommand{\ti}{\tilde}
\newcommand{\Pb}{\mathbb{P}}
\newcommand{\De}{\Delta}
\newcommand{\<}{\langle}
\newcommand{\ra}{\rangle}
\newcommand{\al}{\alpha}
\newcommand{\ga}{\gamma}
\newcommand{\om}{\omega}
\newcommand{\omt}{\tilde{\omega}}
\newcommand{\gat}{\tilde{\gamma}}
\newcommand{\nc}[1]{\mathrm{#1}}
\newcommand{\mc}[1]{\mathcal{#1}}
\newcommand{\mb}[1]{\mathbb{#1}}
\newcommand{\mf}[1]{\mathfrak{#1}}
\newcommand{\dg}[1]{\mathbf{#1}}
\newcommand{\ang}[1]{\big\langle #1\big\rangle}
\newcommand{\nl}{\nolimits}
\newcommand{\dfn}{\leftrightharpoons}
\newcommand{\ran}[1]{\nc{ran}\,{#1}}
\newcommand{\dom}[1]{\nc{dom}\,{#1}}

\newcommand{\Sk}{\operatorname{Sk}}
\newcommand{\Skdir}{\Sk_{\fam0 dir}}
\newcommand{\Skud}{\Sk_{\fam0 ud}}

\def\inserthyphen{\ifcat\next a-\fi\ignorespaces}
\def\pblack-{$\bullet$\futurelet\next\inserthyphen}
\def\pwhite-{$\circ$\futurelet\next\inserthyphen}
\def\pcross-{$\vcenter{\hbox{$\scriptstyle\times$}}$\futurelet\next\inserthyphen}
\def\black{\protect\pblack}
\def\white{\protect\pwhite}
\def\cross{\protect\pcross}

 \section{Introduction}
Hilbert's Sixteenth Problem (Part I) asks to study the topology of a nonsingular real algebraic curve, i.e., the mutual position of its connected components on the plane or an algebraic surface.

Let $S$ be a space of real algebraic curves of fixed degree on the projective plane or a surface, and
$\Delta\subset S$ be a subset of singular curves. The set $\Delta_1\subset
\Delta$ consists of curves with a single non-degenerate double point or a cusp. It is a topological manifold (although not a smooth submanifold of $S$).
The connected components of the set $S\setminus\Delta$ (respectively, $\Delta_1$)
are called \emph{chambers} (respectively, \emph{walls}).

In 1978, V.A. Rokhlin \cite{R} introduced the concept of a rigid isotopy and refined Hilbert's 16th problem by facing the challenge of enumerating the chambers of the space $S$.
Everywhere below, a rigid
isotopy is a path in a chamber or a wall.
\subsection{Principal results}
In the author's paper \cite[Theorem 2]{Z99} it was asserted that {\it a non-singular real curve of bidegree} (4,3)
{\it on a hyperboloid is determined up to a rigid
	isotopy by its complex scheme} (see Figure  \ref{ChW4-3} with the adjacency graph of chambers and walls obtained there).
	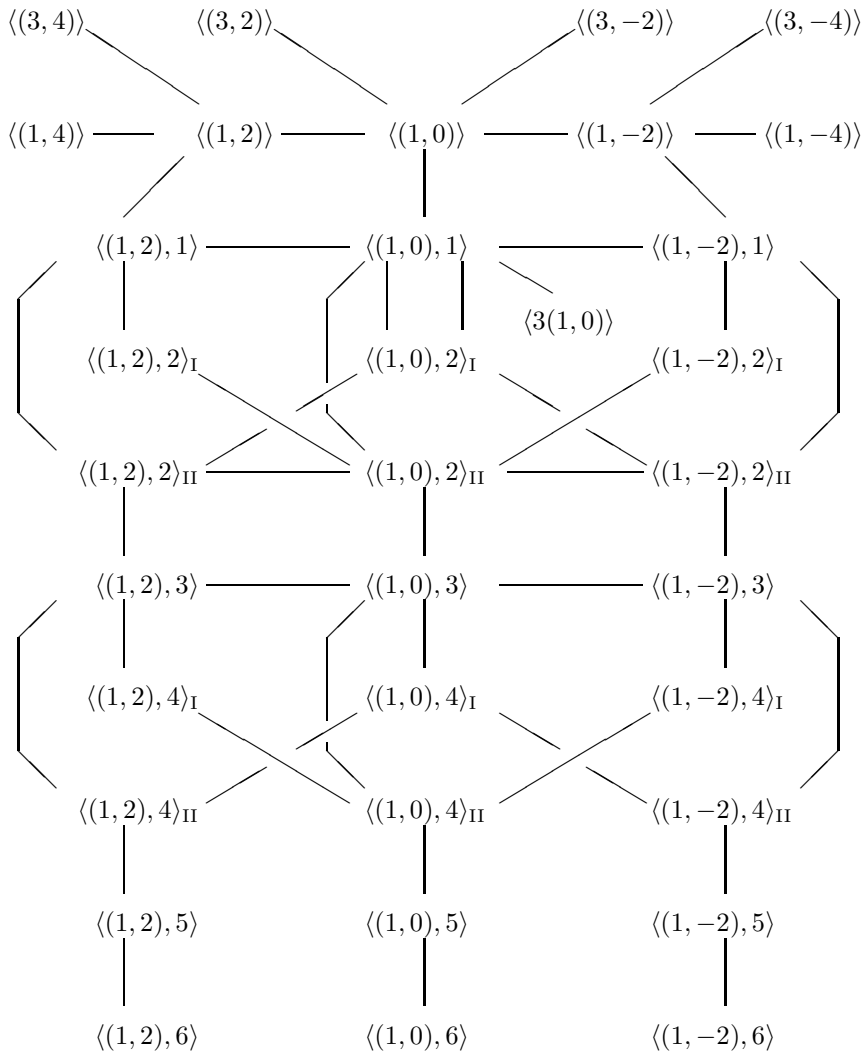
\begin{figure}[h]
		\unitlength=1mm
		\begin{picture}(120,140)(-65,0)
			
			\put(-5,120){$\<(1,0)\ra$}
			\put(-8,105){$\<(1,0),1\ra$}
			\put(-8,90){$\<(1,0),2\ra_{\rm I}$}
			\put(-8,75){$\<(1,0),2\ra_{\rm II}$}
			\put(-8,60){$\<(1,0),3\ra$}
			\put(-8,45){$\<(1,0),4\ra_{\rm I}$}
			\put(-8,30){$\<(1,0),4\ra_{\rm II}$}
			\put(-8,15){$\<(1,0),5\ra$}
			\put(-8,0){$\<(1,0),6\ra$}
			\put(13,95){$\<3(1,0)\ra$}
			\put(10,104){\line(5,-3){7}}
			\put(-8,104){\line(-1,-1){5}}
			\put(-13,99){\line(0,-1){11}}
			\put(-13,85){\line(0,-1){1}}
			\put(-13,84){\line(1,-1){5}}
			\put(-8,59){\line(-1,-1){5}}
			\put(-13,54){\line(0,-1){11}}
			\put(-13,40){\line(0,-1){1}}
			\put(-13,39){\line(1,-1){5}}
			\put(0,5){\line(0,1){9}}
			\put(0,20){\line(0,1){9}}
			\put(0,50){\line(0,1){9}}
			\put(0,65){\line(0,1){9}}
			\put(5,95){\line(0,1){9}}
			\put(-5,95){\line(0,1){9}}
			\put(0,110){\line(0,1){9}}
			
			\put(20,135){$\<(3,-2)\ra$}
			\put(45,135){$\<(3,-4)\ra$}
			\put(20,120){$\<(1,-2)\ra$}
			\put(45,120){$\<(1,-4)\ra$}
			\put(30,105){$\<(1,-2),1\ra$}
			\put(30,90){$\<(1,-2),2\ra_{\rm I}$}
			\put(30,75){$\<(1,-2),2\ra_{\rm II}$}
			\put(30,60){$\<(1,-2),3\ra$}
			\put(30,45){$\<(1,-2),4\ra_{\rm I}$}
			\put(30,30){$\<(1,-2),4\ra_{\rm II}$}
			\put(30,15){$\<(1,-2),5\ra$}
			\put(30,0){$\<(1,-2),6\ra$}
			
			\put(5,125){\line(3,2){15}}
			\put(30,125){\line(3,2){15}}
			\put(8,121){\line(1,0){11}}
			\put(36,121){\line(1,0){8}}
			\put(10,106){\line(1,0){19}}
			\put(11,76){\line(1,0){18}}
			\put(10,61){\line(1,0){19}}
			
			\put(10,77){\line(5,3){20}}
			\put(10,89){\line(5,-3){8}}
			\put(25,77){\llap{\line(-5,3){8}}}
			\put(10,32){\line(5,3){20}}
			\put(10,44){\line(5,-3){8}}
			\put(25,32){\llap{\line(-5,3){8}}}
			
			\put(50,104){\line(1,-1){5}}
			\put(55,99){\line(0,-1){15}}
			\put(55,84){\line(-1,-1){5}}
			\put(50,59){\line(1,-1){5}}
			\put(55,54){\line(0,-1){15}}
			\put(55,39){\line(-1,-1){5}}
			\put(40,5){\line(0,1){9}}
			\put(40,20){\line(0,1){9}}
			\put(40,50){\line(0,1){9}}
			\put(40,65){\line(0,1){9}}
			\put(40,95){\line(0,1){9}}
			\put(40,110){\line(-1,1){8}}
			
			\put(-20,135){\llap{$\<(3,2)\ra$}}
			\put(-45,135){\llap{$\<(3,4)\ra$}}
			\put(-20,120){\llap{$\<(1,2)\ra$}}
			\put(-45,120){\llap{$\<(1,4)\ra$}}
			\put(-30,105){\llap{$\<(1,2),1\ra$}}
			\put(-30,90){\llap{$\<(1,2),2\ra_{\rm I}$}}
			\put(-30,75){\llap{$\<(1,2),2\ra_{\rm II}$}}
			\put(-30,60){\llap{$\<(1,2),3\ra$}}
			\put(-30,45){\llap{$\<(1,2),4\ra_{\rm I}$}}
			\put(-30,30){\llap{$\<(1,2),4\ra_{\rm II}$}}
			\put(-30,15){\llap{$\<(1,2),5\ra$}}
			\put(-30,0){\llap{$\<(1,2),6\ra$}}
			
			\put(-5,125){\line(-3,2){15}}
			\put(-30,125){\line(-3,2){15}}
			\put(-8,121){\line(-1,0){11}}
			\put(-36,121){\line(-1,0){8}}
			\put(-10,106){\line(-1,0){19}}
			\put(-10,76){\line(-1,0){19}}
			\put(-10,61){\line(-1,0){19}}
			
			\put(-10,77){\line(-5,3){20}}
			\put(-9,89){\line(-5,-3){8}}
			\put(-29,77){\line(5,3){8}}
			\put(-10,32){\line(-5,3){20}}
			\put(-9,44){\line(-5,-3){8}}
			\put(-29,32){\line(5,3){8}}
			
			\put(-49,104){\line(-1,-1){5}}
			\put(-54,99){\line(0,-1){15}}
			\put(-54,84){\line(1,-1){5}}
			\put(-49,59){\line(-1,-1){5}}
			\put(-54,54){\line(0,-1){15}}
			\put(-54,39){\line(1,-1){5}}
			\put(-40,5){\line(0,1){9}}
			\put(-40,20){\line(0,1){9}}
			\put(-40,50){\line(0,1){9}}
			\put(-40,65){\line(0,1){9}}
			\put(-40,95){\line(0,1){9}}
			\put(-40,110){\line(1,1){8}}
			
		\end{picture}
		\caption{Chambers and walls in the space
			of bidegree (4,3) curves on a hyperboloid}\label{ChW4-3}
	\end{figure}			
			
To prove this, it was used the approach proposed in \cite{DIK00}
for obtaining a rigid isotopy classification of plane real quintics. The proof is based on Theorem 1 of \cite{Z99}, which specifies all the connected components of the space
of bidegree (4,3) curves that have a single non-degenerate double point or a
cusp. In the present paper, this Theorem 1 is proved in a different way -- using graphs of real trigonal curves on the Hirzebruch surface $\Sigma_3$ obtained from singular curves of bidegree (4,3). The proof of this theorem was started by the author in \cite{Z99}; in \cite{Z03} a gap in this proof was pointed out, which is eliminated in the present paper.

\subsection{Contents of the paper}
Section \ref{hyp} contains the necessary information on the topology of real algebraic curves, available in \cite{R}, \cite{V}, and in particular, of curves on quadrics (see \cite{Z91}) and of trigonal curves on Hirzebruch surfaces. Section \ref{rev} gives a survey of known results on the rigid isotopy classification of real algebraic curves of small degrees in the plane, quadrics, and Hirzebruch surfaces. Sections \ref{trig}
and \ref{S.rational}, following the papers \cite{DIK08} and \cite{Z21}, recall concepts related to real trigonal curves. In Section \ref{S.skeletons}
the concept of a skeleton, introduced in \cite{Z21} for maximally inflected trigonal curves, is extended to a wider class of curves. In Section \ref{deg9} we recall the definition of birational transformations of Hirzebruch surfaces, which are needed later to relate curves on a hyperboloid to trigonal curves on $\Sigma_3$, describe classes of curves of bidegree (4,3) that have a single non-degenerate double point or a
cusp, and fill the indicated gap, which consisted in the absence of a proof of the fact that  each of the classes of singular curves  $\omega^{\pm}_{inn}$, $\alpha^{\pm}_{lp}\<l\ra, l=0,1,2,4,5$, $ \alpha^{\pm}_{lp}\<3\ra_I$, $\alpha^{\pm}_{lp}\<3\ra_{II}$ (in the notation of \cite{Z99}, \cite{Z03}) on a hyperboloid is connected. The same arguments provide a proof of connectedness for the remaining classes of singular curves. As an application, in Section \ref{gen4}, the well-known rigid isotopy classification of nonsingular real trigonal curves of genus 4 on a quadratic cone is described in terms of  graphs and skeletons of these curves.
\section{Definitions and notation}\label{hyp}
\subsection{Real structure} 
A \emph{real algebraic variety} is a complex algebraic variety $V$ with an anti-holomorphic involution $c = c_V: V \rightarrow V$.
The set of fixed points ${\R} V = \mathrm{Fix}\, c$ is called the \emph{real part} of $V$. A regular morphism $f: V \rightarrow W$ between two real varieties is called \emph{real}, or \emph{equivariant}, if $f\circ c_V=c_W\circ f$.
\subsection{Real and complex schemes}
Let $X$ be the real projective plane or a real algebraic surface. The \emph{real scheme} of a real algebraic curve $C\subset X$ is the mutual arrangement scheme of 
 the components of its real part $\mathbb{R} C$ (real branches  for a
singular curve). An \emph{oval} is a component contractible on $\mathbb{R} X$. Each oval bounds a topological
disk in $\R X$, called the {\it interior} of the oval.

A real algebraic curve
$C$ belongs to \emph{type}
$I$ if the set $ \tilde{C}\setminus\mathbb{R}\tilde{C}$
is disconnected, where $\tilde{C}$ is the normalization of~$C$, and to type
$II$ if it
is connected. If $C$ belongs to type $I$, then $\mathbb{R}\tilde{C}$ divides $\tilde{C}$ into two
halves 
having $\mathbb{R}\tilde{C}$ as
their common boundary. The  
complex
orientations of the halves induces two opposite orientations on $\mathbb{R}\tilde{C}$ and, thus, on $\mathbb{R}C$,
called  \emph{complex orientations} of $C$. A real scheme endowed
with a type and, in the case of type I, with complex orientations, is called a {\it
	complex scheme}.
	
\subsection{Real quadrics} 
It is well known that a nonsingular real quadric $X\subset \Pb^3$ is isomorphic to $\Pb^1\times \Pb^1$, and the complex conjugation either preserves or permutes the factors. In the first case, the quadric is a \emph{hyperboloid} with the real part $\R X=\R \Pb^1\times \R \Pb^1$, and in the second case, an \emph{ellipsoid} with the real part $\R X\cong S^2$.

Fix a pair of lines $P_1, P_2$ that generate the quadric $X$. The fundamental classes
$[ P_1]$, $[ P_2]$ form a basis for the group $H_2(X)\cong \Z \oplus \Z$.
Let $C$ be an algebraic curve on $X$. Then $[ C]=
m_2 [ P_1]+m_1 [ P_2]$ for some non-negative integers $m_1$, $m_2$.
The pair $(m_1,m_2)$ is called the {\it bidegree} of  $C$. If $[x_0:x_1]$,
$[y_0:y_1]$ are homogeneous coordinates on the lines $P_1$, $P_2$, then the curve $C$
is defined by a  polynomial
$$F(x_0,x_1;y_0,y_1)=\sum_{i,j=1}^{m_1,m_2} a_{i,j}x_1^i x_0^{m_1-i}y_1^j y_0
^{m_2-j}, $$ that is homogeneous both in $x_0, x_1$ and  $y_0, y_1$
of degrees  $m_1$ and $m_2$ respectively.
On a hyperboloid, the curve $C$ is real
if and only if all $a_{i,j}$ can be chosen to be real. On an ellipsoid,
$C$ is real iff it is possible to choose a polynomial $F$ for which $a_{ij}=\bar{a}_{ji}$ (in particular, $m_1=m_2$).	
\subsection{Coding  a real scheme}
Let $C$ be a nonsingular
real algebraic curve on a hyperboloid $X$. The real part $\R C$ can have components
of two types: contractible in $\R X$ (i.e., ovals) and noncontractible. 
The number of ovals is denoted by $l$, and the number
of noncontractible components by $h$. 
The fundamental classes $[\R P_1]$, $[\R P_2]$,
endowed with certain (fixed) orientations, form a basis for the group
$H_1(\R X)\cong
\Z \oplus \Z$. All noncontractible components $N_1,...,N_h$ realize the same nonzero class $(c_1,c_2)$ in $H_1(\R X)$, where $c_1$ and $c_2$ are coprime. The real scheme of~$\R C\subset \R X$ is encoded as follows:
$\langle (c_1,c_2), {\rm scheme}_1,(c_1,c_2), {\rm scheme}_2,...,(c_1,c_2), {\rm scheme}_h
\rangle,$
where ${\rm scheme}_1,..., {\rm scheme}_h$ are the arrangement schemes of ovals lying in the connected
components of~$\R X\setminus (N_1\cup ... \cup N_h)$ (cf. \cite{V}, \cite{Z91}).

If $X$ is an ellipsoid, then all components of~$\R C$ are ovals; their number is denoted by $l$. In this case, we
fix a point $\infty
\in \R X\setminus \R C$, called the {\it exterior point}, and for the oval
$\omega\subset \R X$, we define its {\it interior} as the component of the surface
$\R X\setminus \omega$ that does not contain $\infty$. This gives rise to a natural partial order on the set of ovals, and the arrangement of ovals
is encoded in the same way as in \cite{V}.

If we need to specify the type of a curve with a real scheme
$\<B\ra$  we write $\<B\ra_{\rm I}$ and
$\<B\ra_{\rm II}$.
\subsection{Trigonal curves on Hirzebruch surfaces}
The Hirzebruch surface $\Sigma_k, k\geq0$, is the space of the fibration $q:\Sigma_k\rightarrow \Pb^1$ with fiber $\Pb^1$, i.e., a rational ruled surface, with exceptional
section $E_k$, $E_k^2=-k$. The fibers of the fibration~$q$ are called
\emph{vertical}; for example, we speak of vertical tangents, vertical
inflections, etc. The surface $\Sigma_0=\Pb^1\times\Pb^1$ is a hyperboloid, in which any curve of bidegree $(0,1)$ can be taken as an exceptional section.

Further, $\Sigma_k$ is considered a real surface with exceptional real
section.

A \emph{trigonal curve} is a reduced curve $C\subset\Sigma_k$ that contains neither the exceptional section nor a fiber
as a component and such that the restriction $q|_C$ is a mapping of degree three. A trigonal curve $C$ is called \emph{proper} if
$C~\cap~E_k=~\varnothing$.

According to \cite[p. 3.1.1]{Degt}, for a proper trigonal curve $C\subset\Sigma_k, k\geq1,$ there exists an affine chart $(x,y)$ in which the exceptional section $E_k$ is defined by the equation $y=\infty$, and the curve $C$ is defined by the Weierstrass equation

$$\,y^3+b(x)y+w(x)=0,$$
where $b, w$ are polynomials, $\deg b \leq 2k, \deg w \leq 3k$. We can extend the chart to $\Sigma_k\smallsetminus E_k$ by considering $b, w$ to be homogeneous polynomials of degrees $2k, 3k$.

A real proper trigonal curve $ C $ is called \emph{almost generic}
if all critical points of the restriction $ q|_C $, i.e., all roots of the discriminant
$ d(x) = 4b^3 + 27w^2 $, are simple (so $ C $ is nonsingular), and \emph{maximally inflected} if all roots of the discriminant are real (and not necessarily simple).
\subsection{Deformations}\label{deformation}
A \emph{deformation} of a trigonal curve
$C\subset\Sigma_k$ is a deformation
of the pair $(q:\Sigma_k\rightarrow \Pb^1, C)$
in the Kodaira-Spencer sense.
A deformation of an almost generic curve
is called \emph{fiberwise} if the curve remains
almost generic throughout the deformation.

The \emph{deformation equivalence} of real trigonal curves is defined by
the equivalence relation generated by equivariant
fiberwise deformations and
real isomorphisms.
\subsection{Trinomial and tetragonal curves}
A generalization of the proper trigonal curve is the \textit{trinomial curve} -- an algebraic curve on the Hirzebruch surface $\Sigma_k$, defined by the equation
\begin{equation}
	y^n+b(x)y^{n-k}+w(x)=0,\label{1}
\end{equation}
where $b, w$ are homogeneous polynomials of degrees $k(n-1)$, $kn$ with $n>1$.

A real \emph{tetragonal} curve $C$ on
$\Sigma_k$ is defined 
by the equation

$$\,y^4+a(x)y^2+b(x)y+w(x)=0,$$

where $a, b, w$ are homogeneous polynomials, $\deg a=2k$, $\deg b = 3k, \deg w = 4k$.
\section{Review of Known Results on Rigid Isotopy Classification}\label{rev}

By now, a rigid isotopy classification has been obtained:
for nonsingular plane curves of degree $m\leq 6$, see \cite{R}, \cite{DIK00}, \cite{Nik}; for plane curves
of degree $6$ with one non-degenerate double point \cite{It91}, \cite{It92}; for
curves of bidegrees $(m,1)$, $(m,2)$, $(3,3)$, $(4,4)$ on a hyperboloid and an
ellipsoid \cite{DK00}, \cite{Z96}, \cite{DZ}, \cite{NS05}; for hyperelliptic curves on Hirzebruch surfaces $\Sigma_k$, see \cite{DIK00}; for nonsingular curves of bidegree $(m,3)$ on Hirzebruch surface $\Sigma_1$, see \cite{Z00}; for tetragonal curves on $\Sigma_1$, $\Sigma_2$, $\Sigma_4$, see \cite{NS05}, \cite{NS07}; 
for real trinomial curves on $\Sigma_k$
with the maximal number of ovals \cite{Z06}.

Let $S_m$ denote the space of plane real algebraic curves of degree $m$, and $S_{m,n}$ the space of real algebraic curves of bidegree $(m,n)$ on a quadric. Figures \ref{ChWC1-5}, \ref{ChWC33h}, and \ref{ChWC33e} show the adjacency graphs of the chambers and walls of the spaces $S_1$ -- $S_5$ and $S_{3,3}$ (for a hyperboloid and an ellipsoid).
\begin{figure}[h]
	\begin{center}
		\scalebox{1}{\includegraphics{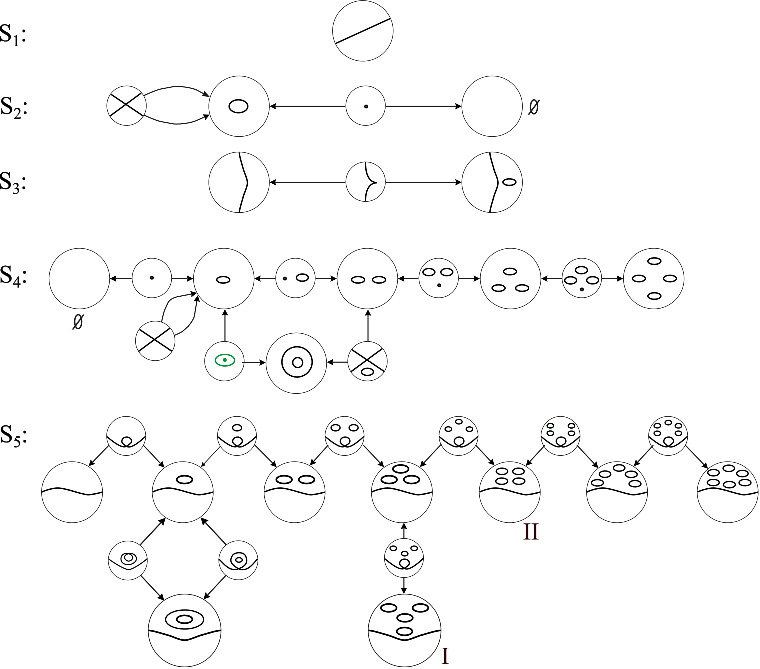}}\\
	\end{center}
	\caption{Adjacency graphs of the chambers and walls in the spaces of plane curves of degrees 1 -- 5}\label{ChWC1-5}
\end{figure}
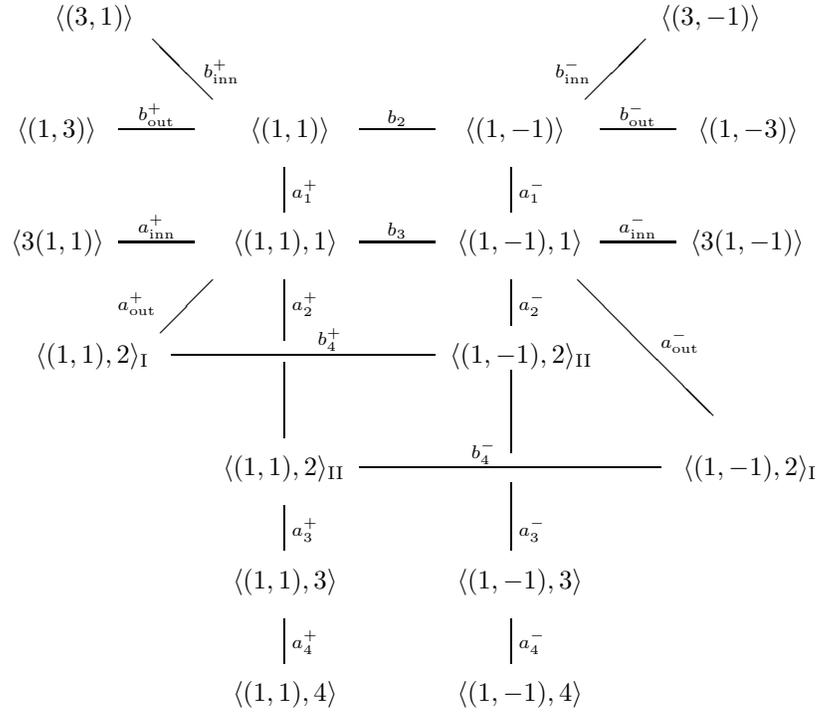
\begin{figure}[h]
	\unitlength=1mm
	\begin{picture}(110,100)(-65,0)
		
		\put(35,90){$\<(3,-1)\ra$}
		\put(9,75){$\<(1,-1)\ra$}
		\put(40,75){$\<(1,-3)\ra$}
		\put(8,60){$\<(1,-1),1\ra$}
		\put(39,60){$\<3(1,-1)\ra$}
		\put(7,45){$\<(1,-1),2\ra_{\rm II}$}
		\put(8,15){$\<(1,-1),3\ra$}
		\put(8,0){$\<(1,-1),4\ra$}
		\put(38,30){$\<(1,-1),2\ra_{\rm I}$}
		
		\put(24,56){\line(1,-1){18}}
		\put(25,80){\line(1,1){8}}
		
		\put(35,47){${\scriptstyle a^-_{\rm out}}$}
		\put(21,83){${\scriptstyle b^-_{\rm inn}}$}
		
		\put(27,76){$\stackrel{b^-_{\rm out}}{\rule{10mm}{.4pt}}$}
		\put(-5,76){$\stackrel{b_2}{\rule{10mm}{.4pt}}$}
		\put(27,61){$\stackrel{a^-_{\rm inn}}{\rule{10mm}{.4pt}}$}
		\put(-5,61){$\stackrel{b_3}{\rule{10mm}{.4pt}}$}
		\put(-5,31){$\stackrel{b^-_4\qquad}{\rule{40mm}{.4pt}}$}
		
		\put(15,67){$\rule[-2mm]{.4pt}{6mm}\lefteqn{\;{\scriptstyle a^-_1}}$}
		\put(15,52){$\rule[-2mm]{.4pt}{6mm}\lefteqn{\;{\scriptstyle a^-_2}}$}
		\put(15,7){$\rule[-2mm]{.4pt}{6mm}\lefteqn{\;{\scriptstyle a^-_4}}$}
		\put(15,22){$\rule[-2mm]{.4pt}{9mm}\lefteqn{\;{\scriptstyle a^-_3}}$}
		\put(15,33){$\rule{.4pt}{11mm}$}

		\put(-35,90){\llap{$\<(3,1)\ra$}}
		\put(-9,75){\llap{$\<(1,1)\ra$}}
		\put(-40,75){\llap{$\<(1,3)\ra$}}
		\put(-8,60){\llap{$\<(1,1),1\ra$}}
		\put(-39,60){\llap{$\<3(1,1)\ra$}}
		\put(-7,30){\llap{$\<(1,1),2\ra_{\rm II}$}}
		\put(-8,15){\llap{$\<(1,1),3\ra$}}
		\put(-8,0){\llap{$\<(1,1),4\ra$}}
		\put(-33,45){\llap{$\<(1,1),2\ra_{\rm I}$}}
		
		\put(-28,56){\llap{\line(-1,-1){7}}}
		\put(-29,80){\llap{\line(-1,1){8}}}
		
		\put(-32,52){\llap{${\scriptstyle a^+_{\rm out}}$}}
		\put(-21,83){\llap{${\scriptstyle b^+_{\rm inn}}$}}
		
		\put(-27,76){\llap{$\stackrel{b^+_{\rm out}}{\rule{10mm}{.4pt}}$}}
		\put(-27,61){\llap{$\stackrel{a^+_{\rm inn}}{\rule{10mm}{.4pt}}$}}
		\put(5,46){\llap{$\stackrel{\qquad b^+_4}{\rule{35mm}{.4pt}}$}}
		
		\put(-15,67){\llap{$\rule[-2mm]{.4pt}{6mm}\lefteqn{\;{\scriptstyle a^+_1}}$}}
		\put(-15,52){\llap{$\rule[-4mm]{.4pt}{8mm}\lefteqn{\;{\scriptstyle a^+_2}}$}}
		\put(-15,7){\llap{$\rule[-2mm]{.4pt}{6mm}\lefteqn{\;{\scriptstyle a^+_4}}$}}
		\put(-15,22){\llap{$\rule[-2mm]{.4pt}{6mm}\lefteqn{\;{\scriptstyle a^+_3}}$}}
		\put(-15,35){\llap{$\rule{.4pt}{10mm}$}}
		
	\end{picture}
	\caption{Adjacency graph of chambers in the space of curves of bidegree (3,3) on a hyperboloid}\label{ChWC33h}
	\end{figure}
	\begin{figure}[h!]
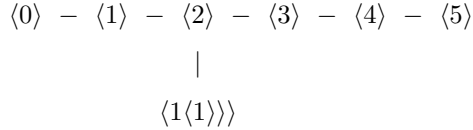

	$$
	\begin{array}{ccccccccccc}
		\langle 0\rangle&-&
		\langle 1\rangle&-&
		\langle 2\rangle&-&
		\langle 3\rangle&-&
		\langle 4\rangle&-&
		\langle 5\rangle\\
		&&&&|&&&&&&\\
		&&&\multicolumn{3}{c}{\langle 1\langle 1\rangle\rangle\rangle}
	\end{array}
	$$
	\caption{Adjacency graph of chambers in the space of curves of bidegree (3,3) on an ellipsoid}\label{ChWC33e}
	\end{figure}
	
	For nonsingular real trigonal curves
	on a ruled surface
	over a base of any genus,
	equivariant deformations were studied and an exact description
	of the deformation classes of curves  with the maximal and premaximal number of ovals (see \cite{DIK00}) and of type I curves (see \cite{DIZ}) was obtained in terms of special type graphs.  
	
	\section{Real Trigonal Curves}\label{trig}
	For a real trigonal curve $ C$, the rational function $j=j_C=\frac{4b^3}{\Delta}=1-\frac {27w^2}{\Delta}$ is called the \emph{$j$-invariant} of this curve. A curve with a constant $j$-invariant is called \emph{isotrivial}.
	An almost generic curve is called \emph{generic} if, for each real critical value $t$ of its $j$-invariant, the multiplicity of each root of the equation ~$j(x)=t$ is~$3$ for $t=0$, is~$2$ for $t\neq 0$, and all these roots are real for $t\in \R\smallsetminus\{0,1\}$. Any almost generic trigonal curve can be made generic by a small change in the coefficients of the curve's equation.
	
	A real  trigonal curve $ C $ (possibly singular) is called
	\emph{hyperbolic} if the restriction ${\R} C \rightarrow {\R} \Pb^1$ of the mapping~$q$ is a three-sheeted covering on the set of non-singular points of the curve.
	
	The real part of a nonsingular nonhyperbolic curve $C$ has a unique
	\emph{long component} $L$, characterized by the fact that the restriction $L \rightarrow {\R} \Pb^1$ of $q$
	has the degree~$\pm 1$. For all other components of ${\R} C$,
	called \emph{ovals}, this degree is equal to~$0$. Let $Z \subset {\R} \Pb^1$ be the set of points
	with more than one preimage in ${\R} C$. Each oval is mapped by  $q$
	to a component of $Z$, which is also called an
	\emph{oval}. The remaining components of $Z$,
	as well as
	their preimages
	in~$L$, are called \emph{zigzags}.
	\subsection{Dessins}\label{graph}
	Below, we need graphs on a disk (\emph{dessins}, as special cases of trichotomic graphs, in the terminology of \cite{Degt},
	\cite{JP}),
	isomorphic to the dessins of real trigonal curves (see the definition in the next paragraph). According to \cite[Corollary 4.13]{Degt}, any such graph is the graph of some trigonal curve. Unless otherwise stated, throughout this section, the term \emph{dessin} implies the existence of a real trigonal curve with such a dessin.
	
	Let $ D $ be the disk obtained as the quotient of the complex projective line $\Pb^1 $ by complex conjugation, and $\mathrm{pr}: \Pb^1 \rightarrow D $ be the projection. Points, segments, etc., lying
	on the boundary circle~$\partial D$ are called \emph{real}.
	For the $ j $-invariant $j_C:\Pb^1\rightarrow \Pb^1= \Cb\cup\{\infty\}$ of a nonisotrivial real trigonal curve $ C\subset\Sigma_k $, we endow the line $ \R \Pb^1$ lying in the image of this function with an orientation determined by the order in $\R$ and color it as follows: let $0$, $1$,
	and~$\infty$ be, respectively, the \black--, \white--, and
	\cross-- vertex; $(\infty,0)$, $(0,1)$, and $(1,\infty)$ be, respectively, a \rm{solid}, \rm{bold}, and \rm{dotted} edge.
	Lifting this orientation and coloring to the graph $\Gamma_C=\mathrm{pr}(j_C^{-1}({\R} \Pb^1))$, we obtain the \emph{dessin} of  $C$. Its \black--, \white--, and
	\cross--vertices, which are branch points (critical points of the $ j $-invariant) with critical values 0, 1 and $\infty $, are called \emph{essential}; the remaining vertices, which are
	branch points with real critical values other than 0, 1, $\infty$, are called \emph{monochrome}.
	Monochrome vertices
	are classified as solid, bold, and dotted
	according to the edges that adjoin them.
	A \emph{monochrome cycle} in~$\Gamma_C$ is a cycle all of whose vertices
	are monochrome; therefore, all of its edges and vertices are of the same color.
	The definition of dessin implies
	that it has no directed monochrome cycles.
	
	Glue two copies of  $D$ into a sphere $S$, turning the disks into hemispheres. Let $p:S\rightarrow D$ be the projection identifying the copies, and $\Gamma'_C=p^{-1}(\Gamma_C)$ be the graph on the sphere with the coloring induced by the projection. The \emph{full valency} of a vertex $v$ in the dessin $\Gamma_C$ is the valency of any vertex $v'\in p^{-1}(v)$ in the graph $\Gamma'_C$. The degree of the mapping $j_C$ is $6k, k\in \N$, so the sum of the full valencies of all \black--, \white--, or
	\cross--vertices of the dessin~$\Gamma'_C$ is $12k$. The number $3k$ is called the \emph{degree} of the dessin $\Gamma_C$; it is equal to the degree of the polynomial $w$ in the equation of the curve $C$. A dessin of degree~$3$ is called \emph{cubic}.
	
	A \emph{singular vertex} of a dessin is a \cross--vertex whose total valency is greater than two. It corresponds to a singular point of the curve.
	
	In the figures, the \emph{real part} $ \partial D \cap\Gamma $ of the dessin $\Gamma $ and its subsets
	are indicated by wide gray lines.
	
	For a dessin
	$\Gamma\subset D$
	the closures of the connected components of~$D\smallsetminus\Gamma$
	are called \emph{regions} of~$\Gamma$. A region with three essential vertices on its
	boundary is called \emph{triangular}.
	
	A dessin  is called
	\emph{unramified} if all its \cross--vertices
	are real.
	In other words, dessins corresponding to
	maximally inflected curves are unramified.
	
	\subsection{Graph Segments}\label{pillar}
	A dessin $\Gamma$
	is called
	\emph{hyperbolic} if all its real edges are dotted. It corresponds to a hyperbolic curve.
	
	For a dessin $\Gamma$, the union of the closures of certain identically colored real edges   is called a \emph{segment} if it is
	homeomorphic to a segment.
	A dotted (bold)
	segment
	is called \emph{maximal}
	if its endpoints are
	two non-singular
	\cross--vertices
	(respectively, two \black--vertices); in this case, the dotted segment must not contain any singular vertices.
	Recall (see the beginning of Section \ref{trig}) that a maximal dotted segment of a non-hyperbolic dessin $\Gamma$ containing an even/odd number of \white--vertices is called an oval/zigzag; it is the projection of an oval/zigzag of the corresponding trigonal curve.
	
	A maximal dotted/bold segment
	with an even/odd number of
	\white--vertices
	is called a
	\emph{wave/jump}.
	\subsection{Elementary moves of dessins}\label{graphmodif}
	Two dessins are said to be
	\emph{equivalent} if, up to a homeomorphism $f$
	of the disk $D$, they can be connected
	by a finite sequence
	of isotopies and the following \emph{elementary moves}:
	\begin{itemize}
		\item[--]
		\emph{monochrome modification}, see Figure
		\ref{fig.moves}(a);
		\item[--]
		\emph{creating} (\emph{destroying}) of a bridge\emph{}, see Figure
		\ref{fig.moves}(b),
		where a \emph{bridge} is a pair
		of monochrome vertices connected by a real monochrome edge;
		\item[--]
		\emph{\white-in} and its inverse \emph{\white-out}, see Figure
		\ref{fig.moves}(c) and~(d);
		\item[--]
		\emph{\black-in} and its inverse \emph{\black-out}, see Figure
		\ref{fig.moves}(e) and~(f);
		\item[--]
		\emph{\cross-in} and its inverse \emph{\cross-out}, see Figure
		\ref{fig.moves}(g) and~(h).
	\end{itemize}
	(In the first two cases, the move
	is considered valid only if it results
	in a graph without directed monochrome cycles, i.e., again a dessin.)
	An equivalence of two dessins is called \emph{restricted} if  $f={\rm id}_D$ and the above isotopies preserve ovals, zigzags, waves, and jumps as sets.
	
	\begin{figure}[h]
		\begin{center}
			\includegraphics{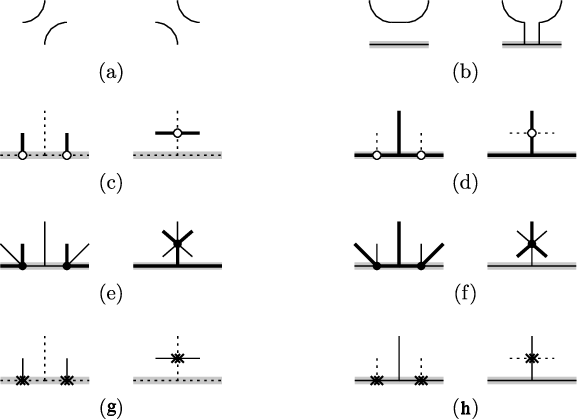}\\
		\end{center}
		\caption{Elementary  moves of dessins}\label{fig.moves}
	\end{figure}
	\subsection{Rigid Isotopies and Weak Equivalence}\label{s.rigid}
	Elementary  moves of a dessin do not allow merging vertical tangents of a trigonal curve, but such merging is possible under rigid isotopies of nonsingular curves. Therefore, to these moves, it is necessary to add
	a pair of mutually inverse operations:
	straightening/creating a zigzag, the first of which consists of
	merging the two vertical tangents
	bounding the zigzag
	into a single tangent at the inflection point
	and then turning them into a pair of complex conjugate imaginary fibers of the bundle $q$. At the dessin level,
	these operations are shown in Figure \ref{fig.zigzag}.
	
	\begin{definition}
		{\rm Following \cite{JP}, we call two dessins \emph{weakly equivalent} if they
		are related by a sequence of isotopies,
		elementary moves (see Figure~\ref{graphmodif})
		and the operation of
		\emph{straightening/creating a zigzag}
		which consists of replacing one of the fragments shown in
		Figure \ref{fig.zigzag} with another.}
	\end{definition}
	
	\begin{figure}[h]
		\begin{center}
			\includegraphics{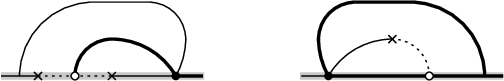}\\
		\end{center}
		\caption{Straightening/creating a zigzag}\label{fig.zigzag}
	\end{figure}
	The following statement is easily deduced from~\cite{DIK08}.
	
	\begin{proposition} \label{equiv.zigzag}
		Two generic real trigonal curves are rigidly isotopic
		if and only if their dessins are weakly equivalent.
	\end{proposition}
	
	As mentioned in \cite{DIZ}, the following theorem
	can be deduced, for example, from~\cite[Proposition 5.5.3, Proposition 5.6.4]{DIK08},
	see also~\cite{Z06}.
	\begin{theorem} \label{max.inflected}
		Any nonhyperbolic nonsingular
		real trigonal curve
		on a Hirzebruch surface
		is rigidly isotopic to a maximally inflected one.
	\end{theorem}
	Within the framework of weak  equivalence of dessins, we need the following two transformations that reduce the number of real \white--vertices in the dessin:
	\begin{enumerate}
		\item[--] removal of neighboring jumps and zigzags, consisting of straightening the zigzag followed by \white-- and \black-in, as well as the inverse transformation (see Figure \ref{JZ});
		\item[--] removal of a pair of neighboring zigzags, consisting of \black-out, straightening two zigzags followed by \white-- and \black-in, as well as the inverse transformation (see Figure \ref{ZZ}).
	\end{enumerate}
	\begin{figure}
		\begin{center}
			\includegraphics{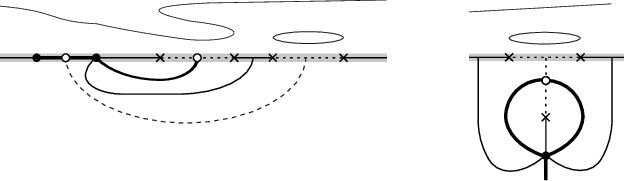}\\
		\end{center}
		\caption{Removing/creating a pair of neighboring jump and zigzag}\label{JZ}
	\end{figure}
	\begin{figure}
		\begin{center}
			\includegraphics{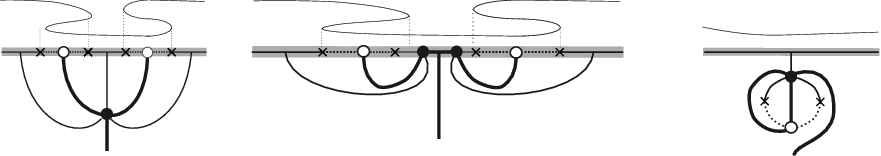}\\
		\end{center}
		\caption{Removing/creating a pair of neighboring zigzags}\label{ZZ}
	\end{figure}
	\subsection{Singular Trigonal Curves}
	A non-isotrivial trigonal curve is called \emph{nodal-cuspidal} if all roots of its discriminant $d(x)$ have multiplicity at most three.
	
	It is easy to verify that \cite[Theorem 3]{Z21} extends to real nodal-cuspidal curves without imaginary singularities:
	\begin{theorem} \label{s.max.inflected}
		Any real nodal-cuspidal non-hyperbolic curve without imaginary singularities is rigidly isotopic to a maximally inflected one.
	\end{theorem}
		
		Real singular \cross-vertices of full valency 4 that are adjacent to solid edges, i.e., corresponding to isolated non-degenerate double (\emph{solitary}) points of the curve ("degenerate ovals"),
	we will call \textit{solitary} and include them among the maximal dotted segments.
	
	\begin{figure}
		\begin{center}
			\scalebox{1}{\includegraphics{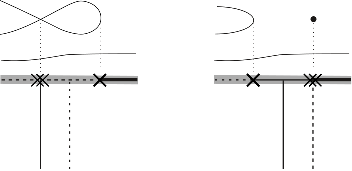}}\\
		\end{center}
		\caption{Passing a nodal point through a cusp}\label{isol}
	\end{figure}
	\begin{proposition}\label{s-oval}
		On a real segment without \white--vertices, a solitary vertex can be swapped with an oval.
	\end{proposition}
	\begin{proof}
		On such a segment, all \black--vertices occur in pairs, which can be removed from the segment using \black-in. After this, the transformation of Figure \ref{isol} swaps the solitary vertex and the oval.
	\end{proof}
	
	\subsection{Cuts}\label{cuts}
	
	For $j=1,2$, let $D_j$ be a disk, $\Gamma_j\subset D_j$ be a dessin of a nodal-cuspidal curve, $I_j\subset\partial D_j$ be a segment, and
	$\varphi:I_1\to I_2$ be an isomorphism, that is, a diffeomorphism of the segments
	establishing an isomorphism of the graphs
	$\Gamma_1\cap I_1\to\Gamma_2\cap I_2$.
	Consider the quotient set $D_{\varphi}=(D_1\sqcup D_2)/\{x\sim\varphi(x)\}$ and the image
	$\Gamma'_{\varphi}\subset D_{\varphi}$ of~$\Gamma_1\cup\Gamma_2$. Denote by~$\Gamma_{\varphi}$
	the dessin obtained from~$\Gamma'_{\varphi}$ by removing the image of the segment
	~$I_1$ if $\varphi$ changes the orientation, and if it does not, then either by transforming
	the images of the endpoints of~$I_1$ into monochrome vertices or by preserving these endpoints as essential vertices.
	
	In what follows, we always assume that $I_j$ is a part of an edge of ~$\Gamma_j$ (see Figure \ref{ArtCut}), or
	$I_j$ contains one
	\white--vertex, or ends at singular vertices and contains one monochrome vertex (see Figure \ref{GenCut}).
	Up to isotopy, in the second and third cases the isomorphism $\varphi$
	is unique; in the first case, this requires specifying whether $\varphi$ preserves or reverses the orientation.
	If $\Gamma_{\varphi}$ is a  dessin, it is called the result of
	 \emph{gluing} of $\Gamma_1$, $\Gamma_2$ along~$\varphi$.
	The image of ~$I_1$ is called a \emph{cut}
	in~$\Gamma_{\varphi}$.
	A cut
	is called \emph{genuine} (\emph{artificial}) if $\varphi$ preserves
	(respectively, reverses) the orientation; it is called \rm{solid},
	\rm{dotted}, or \rm{bold} depending on the structure of the segment
	$\Gamma\cap I_1$.
	(The terms \rm{dotted} and \rm{bold} are still applied
	to cuts containing a
	\white--vertex.)
	A \emph{junction} is a genuine cut obtained by gluing two dessins
	along isomorphic
	parts of their zigzags (see Figure  \ref{GenCut} left).
	\begin{figure}[ht]
		\begin{center}
			\scalebox{1.3}{\includegraphics{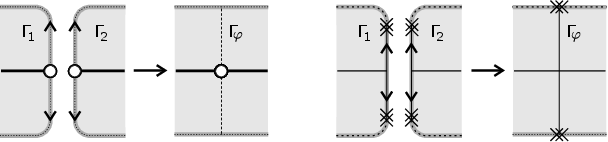}}\\
		\end{center}
		\caption{Examples of genuine gluing}
		\label{GenCut}
	\end{figure}
	
	\begin{figure}[ht]
		\begin{center}
			\scalebox{0.9}{\includegraphics{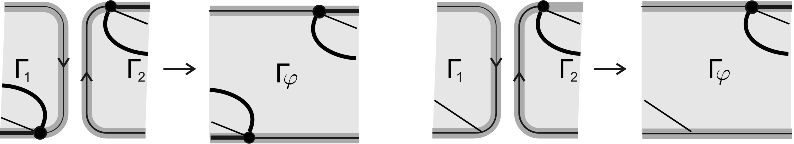}}\\
		\end{center}
		\caption{Examples of artificial gluing}
		\label{ArtCut}
	\end{figure}
	\section{Skeletons}\label{S.skeletons}
	Unramified dessins can be reduced
	to simpler objects, so-called skeletons,
	which are obtained by ignoring all edges except the dotted ones.
	
	Below, the concept of a skeleton, introduced in \cite{Z21} for maximally inflected trigonal curves, is extended to the case of the dessin of a curve obtained from a maximally inflected one by the transformations of Figures \ref{JZ} and \ref{ZZ}. As a result, inner non-isolated white and cross vertices are added to the skeleton
 (see Definition \ref{def.skeleton}).
	
	\subsection{Abstract Skeletons}\label{a.skeletons}
	Consider a (finite) graph $\Sk \subset D$ embedded in a
	disk $D$. We do not exclude the possibility that
	some vertices of~$\Sk$ belong to the boundary of~$D$;
	such vertices are called \emph{real}, the rest are called \emph{imaginary} or \emph{inner}. The set of edges adjacent to each real (respectively, inner)
	vertex~$v$ of~$\Sk$ receives from $D$ a pair of opposite linear
	(respectively, cyclic)
	orders.
	The \emph{immediate neighbors} of an edge~$e$ at~$v$ are the immediate
	predecessor and successor of this edge with respect to (any) of these orders.
	A \emph{first-neighbor path} in~$\Sk$ is a sequence
	of directed
	edges of~$\Sk$ such that
	each edge is followed by one of
	its immediate neighbors.
	
	Below, we consider graphs with
	edges of two types: directed and undirected. We call
	such graphs \emph{partially directed}.
	The directed and undirected parts (the unions of corresponding edges and adjacent vertices) of a partially directed
	graph $\Sk$ are denoted by $\Skdir$ and $\Skud$, respectively.
	
	\begin{definition}\label{def.a.skeleton}
		{\rm Let $D$ be a disk.
		\emph{Abstract skeleton} is a partially directed
		embedded graph $\Sk \subset D$ that is disjoint
		from $\partial D$ except for some
		vertices,
		and satisfies the following conditions:
		\newcounter{N5}
		\begin{list}{(\arabic{N5})}{\usecounter{N5}}
			\item\label{Sk.1}
			each vertex is \emph{white}, \emph{black}, or \emph{cross}; the valency of any inner white vertex is two, any inner black vertex is isolated, any cross vertex is inner monovalent and connected by an incoming edge outgoing from an inner white vertex;
			any edge adjacent to a real black vertex (called the \emph{source})
			is outgoing;
			both edges adjacent to an inner white vertex are outgoing;
			\item \label {Sk.2}
			any immediate neighbor of an incoming
			edge is an outgoing one;
			\item \label {Sk.3}
			$\Sk$ has no first-neighbor cycles;
			\item \label {Sk.4}
			the set of real vertices of the graph ~$\Sk$ is nonempty;
			\item\label{Sk.5}
			for any open region~$R$, i.e. a component of  $D\smallsetminus E_{\Sk}$ with $E_{\Sk}$ being the union of the closures of all edges in $\Sk$, $b_1+3b=v+z+i$, 
			where $b_1$ is the number of black vertices lying on the boundary $\mathrm{Fr}R$ of the region $R$, each of which is incident with one outgoing edge lying on $\mathrm{Fr}R$, $b$ is the number of isolated (real or inner) black vertices in $R$, $v$ is the number of black vertices on $\mathrm{Fr}R$, each of which is incident with two outgoing edges lying on $\mathrm{Fr}R$, $ z $ is the number of connected components of the set $\Skud\cap\,\mathrm{Fr}R$, $i$ is the number of inner white vertices on $\mathrm{Fr}R$  (such a vertex is counted twice if the edges adjacent to it on $\mathrm{Fr}R$ are internal, i.e. $R$ is adjacent to them on both sides).
		\end{list}
		
		If, in addition,
		\newcounter{N6}
		\begin{list}{(\arabic{N6})}{\usecounter{N6}}
			\addtocounter{N6}{\value{N5}}
			\item\label{Sk.8}
			$ \Skdir\cap\Skud=\varnothing $;
			\item\label{Sk.7} each real vertex has no
			directed outgoing edges that are immediate neighbors;
			\item\label{Sk.9} each real white vertex $\Skdir$ has odd valency and is a \emph{sink}, which means that the number of its adjacent incoming edges is one greater than the number of outgoing edges;
			\item\label{Sk.10} each black vertex is real and
			monovalent (i.e., it is a source);
			\item\label{Sk.11}
			the vertices of the subgraphs~$\Skdir$ and $\Skud$ alternate along~$\partial D$,
		\end{list}
		\par\removelastskip
		then $ \Sk $ is called a  \emph{skeleton of type~$I$}.}
	\end{definition}
\subsection{Equivalence of Abstract Skeletons}\label{equivalence}
Two abstract skeletons
are called \emph{equivalent} if, up to a homeomorphism $f:D\rightarrow D$,
they can be connected by the following \emph{elementary moves},
cf.Subsection~\ref{graphmodif}:
\begin{itemize}
	\item[--] \emph{elementary modification}, see Figure \ref{fig.Sk} (a);
	\item[--] \emph{creating} (\emph{destroying})  a \emph{bridge},
	see Figure \ref{fig.Sk} (b);
	the vertex shown in the figure is white;
	other edges of~$\Sk$ may also adjoin the vertex;
	\item[--] \emph{creating} (\emph{deleting})  an \emph{undirected edge},
	see Figure \ref{fig.Sk} (c); the vertex shown in the figure is black and real, the edges adjacent to it are immediate neighbors, and other directed outgoing edges may also adjoin it; the edge on the right side of the figure is undirected;
	\item[--] \emph{\black-in} and its inverse \emph{\black-out}, see Figure
	\ref{fig.Sk} (d), (e); all vertices shown in these figures
	are black, in Figure (d) there may be other directed outgoing edges adjacent to real vertices;
	\item[--] \emph{deleting/creating a neighboring jump and zigzag},
	see Figure \ref{fig.Sk} (f);
	\item[--] \emph{deleting/creating a pair of neighboring zigzags},
	see Figure \ref{fig.Sk} (g);
	\item[--] \emph{transforming a pair of directed edges into an undirected edge and the inverse move},
	see Figure \ref{fig.Sk} (h);
	\item[--] \emph{\cross-in} and its inverse \emph{\cross-out}, see Figure
	\ref{fig.Sk} (i). 
	
\end{itemize}
\begin{figure}[tb]
	\begin{center}
		\includegraphics{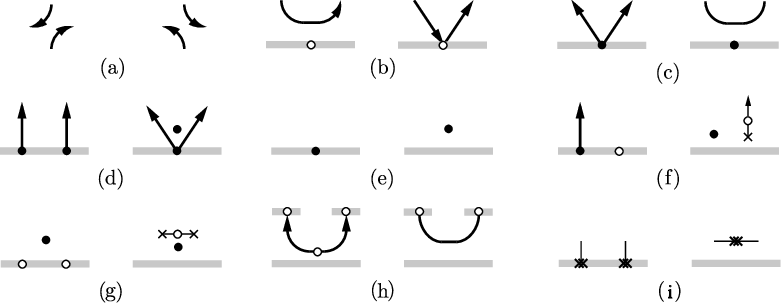}\\
	\end{center}
	\caption{Skeleton moves}\label{fig.Sk}
\end{figure}
(A move is valid only if the result is again an
abstract skeleton.)

An equivalence of two abstract skeletons on the same disk
with the same set of vertices is called \emph{restricted}
if  $f={\rm id}$ and
the above isotopies can be chosen identical on the vertices.
\subsection{Dotted Skeletons}\label{s.skeletons}
Intuitively, a dotted skeleton (see Definition \ref{def.skeleton}) is obtained from
a dessin $\Gamma$ by disregarding all but dotted edges,
patching the latter through all real \white--vertices, and adding inner \black--, \white--, and \cross--vertices.
The undirected edges of the skeleton correspond to the junctions of  $\Gamma$.

\begin{definition}\label{def.skeleton}
	{\rm Let $\Gamma \subset D$ be an unramified dessin or a dessin obtained from an unramified dessin by removing some pairs consisting of neighboring jumps and zigzags, or pairs of neighboring zigzags (see Figures \ref{JZ}, \ref{ZZ}).
	Let $\bar{D}$ be the
	disk obtained from~$ D$ by contracting
	each maximal dotted/bold segment to a point.
	
	The (\emph{dotted}) \emph{skeleton} of~$\Gamma$
	is a partially directed graph
	$\Sk=\Sk_\Gamma\subset \bar{D}$
	obtained from~$\Gamma$ as follows:
	\begin{itemize}
		\item[--]
		each maximal dotted/bold segment is contracted to a point, which is declared a white/black vertex of the skeleton~$\Sk$;
		\item[--] each inner \black--/\cross--vertex of dessin~$\Gamma$ is replaced by an inner black/cross vertex of~$\Sk$;
		\item[--] each inner \white--vertex of~$\Gamma$ that is not on a junction is replaced by an inner white vertex of~$\Sk$, and the dotted edges adjacent to it
		are replaced by directed skeleton edges leading from this white vertex to a cross or white real vertex;
		\item [--]
		each junction is replaced by an undirected skeleton edge, and each of the inner dotted edges of the dessin $\Gamma$ not mentioned above is replaced by a directed skeleton edge with the orientation obtained from the orientation of the dotted edge;
		\item[--] $\Skdir$ ($\Skud$) is the union of black (respectively, white) isolated vertices and the closures of directed (respectively, undirected) edges of~$\Sk$ (so that $\Skdir$ and $\Skud$ may intersect at real vertices).
	\end{itemize}}
	
\end{definition}
The following assertions are proved in the same way as the assertions  \cite[Propositions 5-7, Theorem 2]{Z21} and allow us to talk about the application of transformations of Subsection \ref{equivalence} to dessins.
\begin{proposition} \label{prop.Sk=Sk}
	The skeleton~$\Sk$ of the dessin~$\Gamma$
	from Definition~\ref{def.skeleton}
	is an abstract skeleton in the sense of
	Definition~\ref{def.a.skeleton}. \end{proposition}

\begin{proposition} \label{Sk.extension}
	Any abstract skeleton $\Sk$ is a
	skeleton
	of some dessin~$\Gamma$
	in the sense of Definition~\ref{def.skeleton};
	any two such dessins can be connected by a sequence of
	isotopies and elementary moves,
	see Subsection~\ref{graphmodif},
	that preserve the skeleton.
\end{proposition}

\begin{proposition} \label{prop.Sk}
	Let skeletons $\Sk_1$ and~$\Sk_2$ with the same set of vertices be obtained from dessins $\Gamma_1, \Gamma_2 \subset D$ in accordance with Definition~\ref{def.skeleton}.
	Then $\Gamma_1$ and $\Gamma_2$
	are related by restricted equivalence
	if and only if so are the corresponding
	skeletons~$\Sk_1$ and~$\Sk_2$.
\end{proposition}

\begin{theorem} \label{cor.Sk}
	There is a canonical
	bijection
	between the set of rigid isotopy classes
	of almost generic
	real trigonal curves and the set of equivalence classes
	of abstract skeletons.
\end{theorem}
\section{A Constructive Description of Maximally Inflected Trigonal Curves}\label{S.rational}
This section gives a constructive description
of the real parts of nonsingular maximally inflected
trigonal curves.

\subsection{Blocks}\label{blocks}
Consider a class of unramified dessins defined constructively according to the following definition.
\begin{definition}\label{def.block}
	A cubic block of type I \emph{ is an unramified dessin of degree 3 of type I (see Figure \ref{cubics}~I).} A cubic block of type II \emph{  is an unramified dessin of degree 3 of type II with an inner \black--vertex (see Figure \ref{cubics} II). Several cubic blocks, artificially glued together along segments of solid edges, form a (}general\emph{)} block.
	
	A block type \emph{is the type of the curve corresponding to the block.}
	
\end{definition}
\begin{figure}[h]
	\begin{center}
		\includegraphics{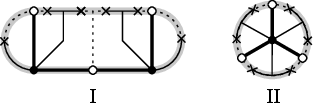}\\
	\end{center}
	\caption{Cubic blocks}\label{cubics}
\end{figure}

According to \cite[5.6.7]{DIK08} and \cite[5.1]{Z21}, both cubic blocks are unique up to isomorphism.

The following proposition describes the blocks.
\begin{proposition}\label{block.existence} {\rm \cite[Proposition 8]{Z21}.} 
	Let $d\ge1$ be an integer, and let $O,J\subset S^1=\partial D$ be two disjoint $d$-element sets. Then there exists a unique, up to
	restricted equivalence, block $\Gamma\subset D$ of type I and degree~$3d$
	with an oval about each
	point of~$O$, a jump at each point of~$J$, and a zigzag between
	any two points of $O\cup J$ \emph{(}and no other dotted or bold segments\emph{)}.
	
	A block of degree~$3d$ of arbitrary type with $ c $ jumps, $ c $ ovals, $ b $ inner \black--vertices, and $ z $ zigzags corresponds to an abstract skeleton in the disk with $ c $ oriented disjoint
	chords, $ b $ inner black vertices, and $ z $ real isolated white vertices that satisfy the following conditions:
	\begin{enumerate}
		\item \label{11.1}
		$b+c=d$, $z+c=3d$;
		\item\label{11.2}
		for each component $ R_i $ of the closed cut  of the disk along the chord, $z_i=c_i+3b_i$, where $c_i$, $b_i$, and $z_i$ are the numbers of chords, black inner vertices
		and real isolated vertices of this component.
	\end{enumerate}
\end{proposition}
\begin{remark}\label{replace}
	From this description,
	it immediately follows that the genuine gluing of two blocks along bold segments can be replaced either by junction between two other blocks, composed of parts of the previous blocks, or by artificial gluing along segments of solid edges, i.e., by a new block.
\end{remark}

\subsection{Real Parts of Maximally Inflected Curves}
\label{CR}
A complete description of the real part of a maximally
inflected nonsingular
trigonal curve,
\emph{i.e.} the description of the topology of the pair
$(\R\Sigma_k,\R C)$ is given in \cite[5.3]{Z21} and, taking into account Remark \ref{replace}, looks as follows:

The dessin of a maximally
inflected curve is obtained from a disjoint
union of blocks using junctions that transform the disks of the blocks into a single
disk.
Moreover, if all blocks are of type I and all gluings are junctions, the resulting dessin is of type I; otherwise, the resulting dessin is of type II.
\subsection{Blocks and Weak Equivalence of Dessins}
By \cite[Proposition 10, Lemma 1]{Z21}, for any~$d\ge1$, there exists a unique, up to
weak equivalence, block $\Gamma\subset D$ of type I and degree~$3d$.
\begin{theorem}\label{contact}
	A block of type II with at least two ovals is weakly equivalent to a dessin with a junction.
\end{theorem}
\begin{proof}
	
	Let $B$ be a block of type II and $d\geq2$ its degree.
	If $B$ has two jumps
	connected by a solid segment, then a junction is obtained after moves of Figures \ref{fig.Sk} (d), (c) for the skeleton of $B$. Otherwise, each jump has its own neighboring zigzag, and all jumps can be eliminated using the move of Figure \ref{fig.Sk} (f). Thus,  according to Proposition \ref{block.existence}, a dessin with $d$ (inner) \black--vertices, $c$ ovals, and $z$ zigzags, $c\leq d$, $z\geq d$ is obtained. If the resulting dessin has the zigzags alternate with the ovals, then the inverse move returns us to a block of type $I$ according to the same proposition.
	
	Therefore, $B$ contains two neighboring ovals, between which (on at least one of the real segments) there are $r\geq2$ zigzags. We apply the move of Figure \ref{fig.Sk} (g) to $[r/2]$ triples consisting of the pair of real white vertices corresponding the  zigzags, and an inner black vertex. Depending on whether $r$ is even or odd, we transform the skeleton of the resulting dessin according to Figure \ref{0} or Figure \ref{1}.
	\begin{figure}[h]
		\begin{center}
			\scalebox{1}{\includegraphics{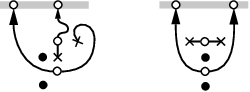}}\\
		\end{center}
		\caption{A junction with a cubic block of type II}\label{0}
	\end{figure}
	
	\begin{figure}[h]
		\begin{center}
			\scalebox{1}{\includegraphics{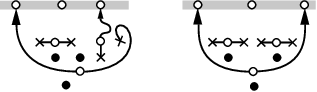}}\\
		\end{center}
		\caption{A junction with a block of degree 6}\label{1}
	\end{figure}
	
\end{proof}
\section{Singular curves of bidegree (4,3) on a hyperboloid and proper trigonal curves}\label{deg9}
\begin{definition}
	{\rm\emph{A positive (negative) Nagata transformation} (see \cite[\S\,2 (3)]{N}) is a fiberwise birational transformation $\Sigma_k\rightarrow\Sigma_{k+ 1}$ (respectively, $\Sigma_k\rightarrow\Sigma_{k-1}$),
	consisting of blowing up a point $p\in E_k$ (respectively, $p\notin E_k$) and then contracting the proper transform
	of the fiber $q^{-1}(p)$ to a point.}
\end{definition}
A positive (negative) Nagata transformation maps $E_k$ to $E_{k+1}$ (respectively, to $E_{k-1}$).

It is well known (see, e.g., \cite[3.1.1]{Degt}) that positive and negative Nagata transformations establish a correspondence between (trigonal) curves of bidegree $(m,3)$ on a hyperboloid and proper trigonal curves: the curve $C\subset\Sigma_0$ corresponds to its image $N(C)\subset\Sigma_m$ under a positive Nagata transformation that blows up the intersection points of $C$ with the curve $E_0\subset\Sigma_0=\Pb^1\times\Pb^1$.

Let us list the walls in the space $S_{4,3}$ and find a correspondence between the rigid isotopies of curves in the wall and the transformations of the dessins of the corresponding proper trigonal curves.
\subsection{Walls in the space $S_{4,3}$}
Figures \ref{omega}, \ref{gamma}, \ref{alpha} indicate, in the notation of \cite{Z99}, the real schemes of curves of classes $\omega^{\pm}_{inn}$, $\omega^{\pm}_{out}$, $\tilde{\omega}$, $\gamma^{\pm}_{inn}$, $\gamma^{\pm}_{out}$, $\tilde{\gamma}$, $\alpha^{\pm}_{lp}\<l\ra$, $\alpha^{\pm}_{ov}\<l\ra$, $\tilde{\alpha}\<l\ra$ in the space $S_{4,3}$, marked with the superscript "$+$" (for curves of classes $\gamma^{+}_{out}$, $\alpha^{+}_{lp}$, $\tilde{\alpha}$, the schemes  with a solitary singular point are shown). The schemes with the superscript "$-$" are obtained by reflection relative to a vertical fiber. The dessins and the skeletons in these figures are described below.

\begin{figure}[h]
	\begin{center}
		\scalebox{0.7}{\includegraphics{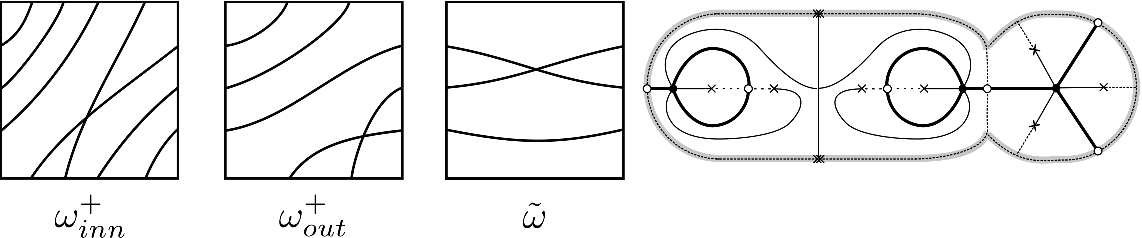}}\\
	\end{center}
	\caption{Real schemes of curves of classes $\omega^{+}_{inn}$, $\omega^{+}_{out}$, $\tilde{\omega}$ and the dessin of a corresponding trigonal curve}\label{omega}
\end{figure}

\begin{figure}[h]
	\begin{center}
		\scalebox{0.8}{\includegraphics{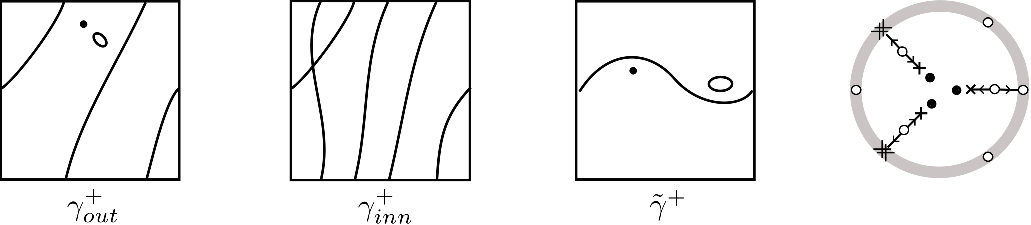}}\\
	\end{center}
	\caption{Real schemes of curves of classes $\gamma^{+}_{inn}$, $\gamma^{+}_{out}$, $\tilde{\gamma}$ and the skeleton of a corresponding trigonal curve}\label{gamma}
\end{figure}
\begin{figure}[h!]
	\begin{center}
		\scalebox{0.8}{\includegraphics{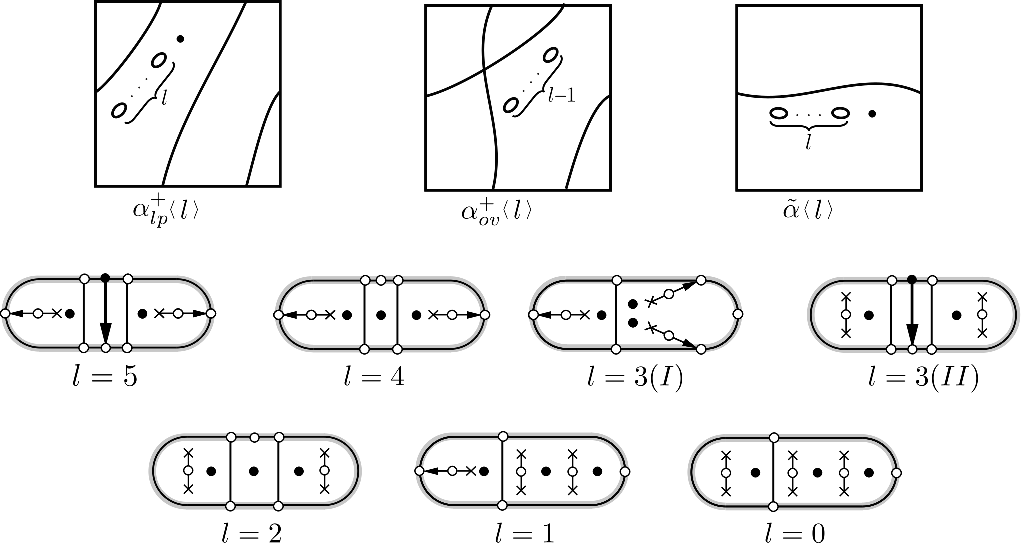}}\\
	\end{center}
	\caption{Real schemes of curves of classes $\alpha^{+}_{lp}\<l\ra$, $\alpha^{+}_{ov}\<l\ra$, $\tilde{\alpha}\<l\ra$ and skeletons of  corresponding trigonal curves}\label{alpha}
\end{figure}
In the paper \cite{Z99} it is proved that each of the complex schemes $\om^{\pm}_{\rm out}$, $\omt$,
$\ga^{\pm}_{\rm inn}$, $\ga^{\pm}_{\rm out}$, $\gat^{\pm}$,
$\al^{\pm}_{\rm ov}B$ $(B=\<1\ra$, $\<2\ra$, $\<3\ra_{I}^{\pm}$,
$\<3\ra_{II}$, $\<4\ra$, $\<5\ra^{\pm})$, $\alt B$ $(B=\<0\ra$, $\<1\ra$,
$\<2\ra$, $\<3\ra_{I}$,
$\<3\ra_{II}$, $\<4\ra$, $\<5\ra)$ corresponds to a single wall. The same remains to be proved for the complex schemes $\{\om^{\pm}_{\rm inn}$, , $\al^{\pm}_{\rm lp}B$ $(B=\<0\ra$, $\<1\ra$, $\<2\ra$, $\<3\ra_{I}$,
$\<3\ra_{II}$, $\<4\ra$, $\<5\ra)$.
\subsection{Singular Fibers and Nagata Transformations}\label{s.fiber}
 For a  trigonal curve $C\subset\Sigma_k$, a fiber of $\Sigma_k$ is called \emph{singular} if it intersects $C\cup E_k$ geometrically in fewer than four points. 
To obtain a proper trigonal curve on $\Sigma_3$ from a singular curve $C\in\De \smallsetminus S$ of bidegree $(4,3)$ on a hyperboloid using a Nagata transformation, we take, as an exceptional section $E_0\subset\Sigma_0$, a curve of bidegree $(0,1)$ passing through a singular point of $C$, and consider  a chart  $(x,y)$ on the hyperboloid, where $E_0$ and the singular fiber $F$ are given by the equations $y=\infty$ and $x=0$. Then the positive Nagata transformation $N$ centered at the point $p=F\cap C\cap E_0$ is given by the equality $(x,z)=(x,xy)$. Below, the curve $C$ and its image $N(C)$ are given by local equations with only the necessary initial terms indicated. For curves from the classes $\omega^{\pm}_{inn}$ and $\alpha^{\pm}_{lp}$, we need the Nagata transformations of the following
singular fibers (the fibers $N(F)$ and $N^2(F)$ of the curves $N(C)$ and $N^2(C)$ are denoted according to \cite[3.1.2]{Degt}):
\begin{enumerate}
	\item\label{s1} $p$ is a non-degenerate double point at which the curve $C$ is tangent to neither $F$ nor $E_0$. $C: x^2y^3+y+1=0$, $N(C): z^3+z+x=0$, fiber $N(F):\,\tilde{A}_0$;
	\item\label{s2} $p$ is a non-degenerate double point at which the curve $C$ is tangent to fiber $F$ and not tangent to $E_0$. $C: x^2y^3+xy^2+1=0$, $N(C): z^3+z^2+x=0$, fiber $N(F):\,\tilde{A}_0^*$;
	\item\label{s3} $p$ is a non-degenerate double point at which the curve $C$ is tangent to $E_0$ and not tangent to $F$. $C: x^3y^3+xy^2+y+1=0$, $N^2(C): z^3+z^2+xz+x^3=0$, fiber $N^2(F):\,\tilde{A}_1$;
	\item\label{s4} $p$ is a non-degenerate double point at which the curve $C$ is tangent to both $E_0$ and $F$. $C: x^3y^3+xy^2+1=0$, $N^2(C): z^3+z^2+x^3=0$, fiber $N^2(F):\,\tilde{A}_2$;
	\item\label{s5} $p$ is a cusp at which the tangent coincides with neither $E_0$ nor $F$.
	$C: x^2y^3+2xy^2+y+x^2y^2+1=0$, $N(C): z^3+2z^2+z+xz^2+x=0$, fiber $N(F):\,\tilde{A}_0^*$;
	\item\label{s6} $p$ is a cusp with a vertical tangent. $C: x^2y^3+1=0$, $N(C): z^3+x=0$, fiber $N(F):\,\tilde{A}_0^{**}$;
	\item\label{s7} $p$ is a cusp with a horizontal tangent. $C: x^3y^3+y+1=0$, $N^2(C): z^3+xz+x^3=0$, fiber $N^2(F):\,\tilde{A}_1^{*}$;
	\item\label{s8} $p$ is a nonsingular point of $C$ with a non-vertical and non-horizontal tangent, and $F$ intersects $C$ in three different points. $C: xy^3+y^2+1=0$, $N(C): z^3+z^2+x^2=0$, fiber $N(F):\,\tilde{A}_1$;
	\item\label{s9} $p$ is a non-singular point of $C$ with a non-vertical and non-horizontal tangent, and $F$ is tangent to $C$ (at a point different from $p$). $C: xy^3+y^2+x=0$, $N(C): z^3+z^2+x^3=0$, fiber $N(F):\,\tilde{A}_2$;
	\item\label{s10} $p$ is a non-singular point of $C$ with a vertical tangent, and $F$ intersects $C$ at a point different from $p$. $C: xy^3+y+x=0$, $N(C): z^3+xz+x^3=0$, fiber $N(F):\,\tilde{A}_1^*$;
	\item\label{s11} $p$ is the inflection point of  $C$ with a vertical tangent. $C: xy^3+1=0$, $N(C): z^3+x^2=0$, fiber $N(F):\,\tilde{A}_2^*$.
	
\end{enumerate}
\subsection{Curves $\omega^{\pm}_{inn}$}\label{w_inn}
A curve $C\in\omega^{\pm}_{inn}$ is hyperbolic. Its real scheme can be obtained by combining the real schemes $\langle(2,3)\rangle $ and $\langle(1,1)\rangle$ of the nonsingular branches of the curve that intersect transversally at a single point $p$ (see Figure \ref{omega}).
Therefore,  $C$
has three singular fibers: two fibers of the form \ref{s8} and a fiber $F_p$ of the form \ref{s1} from Subsection \ref{s.fiber}. Therefore,
the curve $N_3(C)\subset\Sigma_3$, where $N_3$ is the composition of three positive Nagata transformations, has two distinct nodal points and a nonsingular point $a\in N_3(F_p)$ lying on the image of the branch $(2,3)$ (on the arc of the curve bounded by the singular points that has an even intersection multiplicity with the line $y=0$).

\begin{lemma}\label{solidcut}
	The dessin of $N_3(C)$ is equivalent to a dessin with a genuine solid cut connecting an inner solid vertex to the singular vertices.
	
\end{lemma}
\begin{proof}
	Since the degree 9 of the dessin of $N_3(C)$ is odd, it has an odd number of \white--vertices. On a real segment of the dessin bounded by singular vertices $s_1, s_2$ and having an odd number of \white--vertices, we remove all these \white--vertices except one using \white--in and destroy any bridges that may have appeared. The remaining \white--vertex $w_1$ is connected by real edges $\delta_1,\delta_2$ with $s_1, s_2$ and by a bold edge $\beta_1$ with some \black--vertex $b$. Let $\beta_2, \beta_3$ be other bold edges outgoing $b$ (see Figure \ref{s.cut}), and $w_2$ be the \white--vertex that is the endpoint of the edge $\beta_2$. After applying monochrome modifications if necessary, we obtain that $w_2$ is the endpoint of the edge $\beta_3$ and the endpoint of the path $b,\sigma_1,c,\delta_3,w_2$, where $\sigma_1$ is the solid edge lying between $\beta_1$ and $\beta_2$, $c$ is a \cross-vertex, and $\delta_3$ is a dotted edge. We connect vertices $b$ and $s_1$ by a solid edge $\sigma_2$, applying monochrome modifications if necessary. The region of the resulting dessin that contains the path $s_1,\sigma_2,b,\beta_1,\delta_2,s_2,\sigma_3$ on the boundary, where $\sigma_3$ is a solid edge, is not triangular, so during the (solid) monochrome modification of the edges $\sigma_2,\sigma_3$, one can stop at an intermediate position, creating an inner monochrome vertex and, thus, the desired cut.
\end{proof}
\begin{figure}
	\begin{center}
		\scalebox{0.8}{\includegraphics{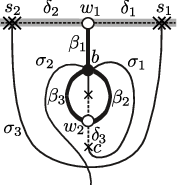}}\\
	\end{center}
	\caption{A solid cut}\label{s.cut}
\end{figure}
\begin{remark}\label{ssp}
	The fiber $N_3(F_p)$ corresponds to a point in the dessin of the curve $N_3(C)$ that lies on the real segment of the dessin bounded by the singular vertices $s_1, s_2$ and having an even number of \white--vertices. Otherwise, the moves of Figure \ref{fig.moves} (c), (g)  maps $s_1, s_2$ to an inner singular vertex, which is impossible, according to the statement at the beginning of  Subsection~\ref{w_inn}.
\end{remark}
\begin{theorem}
	Each of the classes $\omega^{+}_{inn}$ and $\omega^{-}_{inn}$ is connected, i.e., consists of a single wall.
\end{theorem}
\begin{proof}
	After the cut specified in Lemma \ref{solidcut}, we obtain a cubic dessin of type II and a dessin of degree 6 with one oval, which, by Theorem \ref{max.inflected}, are weakly equivalent to the corresponding blocks (see the dessin on Figure \ref{omega}). The latter are unique up to weak equivalence according to \cite[Proposition 8, Lemma 1]{Z21}.
	
	If a curve $C\in\omega^{+}_{inn}$ is constructed from a proper trigonal curve $C'\subset\Sigma_3 $, then $C'=N_3(C)$, and a curve in $\omega^{-}_{inn}$ symmetric to $C$ with respect to $E_0$ can be obtained from a curve that in the affine chart $\Sigma_3\smallsetminus(E_3\cup N_3(F_p))$ is symmetric to $C'$ with respect to the $x$-axis, and obviously has the same dessin as $C'$.
\end{proof}
\begin{remark}
	Curves in the classes $\omega^{\pm}_{out} $ and $\tilde{\omega}$ differ from the curve $C\in\omega^{\pm}_{inn}$ only by the choice of the point $a\in N_3(C)$. Let $C'$ be the arc of the image of the branch $(2,1)$, bounded by singular points and having odd intersection multiplicity with the line $y=0$. The point $a$ lies on  $C'$ for  $C\in\omega^{\pm}_{out}$, and  on the image of the branch $(1,0)$ for  $C\in\tilde{\omega}$. Therefore, the same arguments prove  connectedness of these classes as well.
	
	\end{remark}
\subsection{Curves $\alpha^{\pm}_{lp}\<l\ra$}
A curve $C\in\alpha^{\pm}_{lp}\<l\ra$ is nonhyperbolic. Its real part contains a (possibly singular) branch realizing the class $(1,\pm2)$ in $H_1(\R X)$, and $l$ ovals if the singular point does not lie on an oval (see Figure \ref{alpha}). Therefore, $C$ has either three singular fibers: two fibers of the forms \ref{s8}-\ref{s11} and a fiber $F_p$ of one of the forms \ref{s1}, \ref{s2}, \ref{s5}, \ref{s6} from  Subsection~\ref{s.fiber}, or
two singular fibers: a fiber of one of the forms \ref{s8}-\ref{s11} and a fiber $F_p$ of one of the forms \ref{s3}, \ref{s4}, \ref{s7}.
Consequently, a proper trigonal curve $N_3(C)\subset\Sigma_3$ has the singular fibers listed in Subsection~\ref{s.fiber}. To return to~$C$ using negative Nagata transformations, we need to blow up the singular points of~$N_3(C)$ and the point $a$ lying on the image of the branch $(1,\pm2)$ (on the arc of the curve bounded by the singular points that has an even intersection with the line $y=0$).

\begin{remark}\label{gamma_out}
	For $l=1$, the curve $C$ is of type II. A curve of type I with a similar real scheme {\rm (see Figure \ref{gamma})} belongs to one of the walls of $\gamma^{\pm}_{out}$. The uniqueness of the latter, as well as of the walls $\tilde{\gamma}^{\pm}$, $\gamma^{\pm}_{inn} $, follows from the uniqueness {\rm (see \cite[Proposition 10, Lemma 1]{Z21})} of the type~I block of degree 9,  with the skeleton in Figure {\rm\ref{gamma}}, since these walls differ only in the choice of the point $a\in N_3(C)$. Let $C'$ be an arc of the image of the branch $(1,\pm2)$, bounded by the singular points. For  $\gamma^{\pm}_{out}$/$\tilde{\gamma}^{\pm}$  the point $a$ lies on  $C'$ that has even/odd multiplicity of intersection with the line $y=0$. For  $\gamma^{\pm}_{inn} $ the point $a$ lies on the oval.
\end{remark}

\begin{lemma}\label{isolp-ts}
	There exists a nodal-cuspidal curve with two 
	solitary singular points that is rigidly isotopic to the curve $N_3(C)$.
\end{lemma}
\begin{proof}
	The singular points of $N_3(C)$ correspond to singular vertices of its dessin. Suppose that a singular vertex $s$ is not solitary. If $s$ is connected by a real dotted edge to a (non-singular) \cross-vertex, then $s$ can be made solitary using the transformation of Figure \ref{isol}.
	
	Now suppose that  $s$ is connected with two non-singular \cross-vertices by real dotted segments containing an odd number of \white--vertices. Using \white-in, we leave one \white--vertex on each of these segments.
	By Theorem \ref{s.max.inflected}, the dessin of $N_3(C)$ is unramified, so the result $\Gamma$ of its perturbation that turns singular vertices into ovals,
	breaks into blocks after genuine cuts along the dotted edges (see  Subsection~\ref{CR}). Consequently, the solid segment obtained from $s$ lies in one of the blocks between two zigzags; therefore, the block is of type II and, by Theorem \ref{contact}, has at most one oval. Consequently, a jump lies near one of the zigzags. Returning to the dessin of $N_3(C)$ and applying the transformations of Figures \ref{A1} and \ref{isol} to the fragment of the dessin containing the \black--vertex of this jump, we obtain a solitary singular vertex.
	\begin{figure}
		\begin{center}
				\scalebox{0.8}{\includegraphics{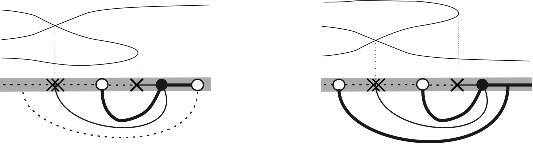}}\\
		\end{center}
		\caption{ A passage of a nodal point through a point $\tilde{A}_1^*$}\label{A1}
	\end{figure}
	Finally, the last possible case that prevents us from applying the transformation of Figure \ref{isol} immediately is when both singular vertices $s_1,s_2$ lie between two (nonsingular) \cross--vertices $v_1, v_2$ on the same dotted segment and split it into three segments, each containing an odd number of \white--vertices (since  Remark \ref{ssp} is obviously true also for the curve $C$, there is an odd number of \white--vertices between $s_1,s_2$). Using \white-in, we leave one \white--vertex on each of these segments. The same arguments as in the previous case show that a jump lies near one of the  \cross--vertices $v_1, v_2$, so applying the transformation of Figure \ref{A1} removes the \white--vertex between $s_1,s_2$, resulting in a dessin of the curve that does not lie in $\alpha^{\pm}_{lp}\<l\ra$ by virtue of Remark \ref{ssp}.
	\end{proof}
	Using the proven lemma, we will further assume that the singular vertices of the dessin $N_3(C)$ correspond to solitary singular points. Denote by $\Gamma$ the result of the perturbation of this dessin that turns the singular vertices into ovals.
	\begin{lemma}\label{1white}
	The dessin of the curve $N_3(C)$ is weakly equivalent to a dessin with a single real \white--vertex.
	\end{lemma}
	\begin{proof}
	By Remark \ref{gamma_out}, the dessin $\Gamma$ cannot be a block of type I, so by Theorem \ref{contact}, the dessin is a junction of three cubic blocks or a cubic block and a block of degree 6 (see Figure \ref{blocks6}). The transformations of Figures \ref{JZ}, \ref{ZZ} in the outer blocks in the first case and in both blocks in the second, followed by contraction of  two ovals to singular points, yield the desired dessin.
	\end{proof}
	\begin{figure}
	\begin{center}
		\scalebox{0.9}{\includegraphics{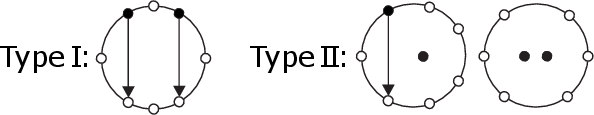}}\\
	\end{center}
	\caption{Blocks of degree 6 up to weak equivalence}\label{blocks6}
	\end{figure}
	\begin{theorem}
	A wall in the class of curves $\alpha^{\pm}_{lp}\<l\ra$ is uniquely determined by a complex scheme.
	\end{theorem}
	\begin{proof}
	After the transformation specified in Lemma \ref{1white}, singular vertices can be interchanged with any ovals according to Proposition \ref{s-oval}. Therefore, it suffices to prove the theorem for the dessin $\Gamma$.
	Consider all values of $l$. The skeletons of the dessins studied below are indicated in Figure \ref{alpha}.
	\begin{enumerate}
		\item For $l=5$, by \cite[Lemma 1]{Z21}, the dessin $\Gamma$ is unique up to weak equivalence as an $M$-curve dessin.
		\item For $l=4$, by \cite[6.4.2]{DIK08} and \cite[Lemma 1]{Z21}, the dessin $\Gamma$ is unique up to weak equivalence as an $(M-1)$-curve dessin.
		\item For $l=3$, as already mentioned in the proof of Lemma \ref{1white}, the dessin $\Gamma$ is a junction of a cubic block $K$ and a dessin $B$ of degree 6. Moreover, if $\Gamma$ is of type~I, then $K$ and $B$ are blocks of type I, unique up to weak equivalence by \cite[Proposition 8, Lemma 1]{Z21}. If $\Gamma$ is of type II, then $B$ contains a junction by Theorem \ref{contact}, otherwise it would be a block of type I with two ovals, and therefore $K$ would also be a block of type I. Therefore, $\Gamma$ is a junction of three cubic blocks, one of type I and two of type II. The union of a block of type I and a block of type II yields an $(M-1)$-curve dessin, in which the blocks can be permuted according to \cite[6.4.2]{DIK08}. Therefore, $\Gamma$ is equivalent to a dessin with a central cubic block of type I and is therefore unique.
		\item For $l=2$, as for $l=3$, $\Gamma$ is a junction of a cubic block $K$ and a dessin $B$ of degree 6. If $K$ is of type I, the following skeleton moves of $\Gamma$
		allow us to obtain a junction with endpoints in neighboring ovals, turning $K$ into a cubic block of type II: moves (f), (g) of Figure \ref{fig.Sk} applied to $K$ and $B$, move (h) of Figure \ref{fig.Sk} applied to the junction, move (a) \ref{fig.Sk} applied to edges $e_1,e_2$ on Figure \ref{l=2}, move of Figure \ref{0} applied to edges $e'_1,e_3$ on Figure \ref{l=2}.
		\begin{figure}[h!]
			\begin{center}
				\scalebox{1.2}{\includegraphics{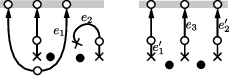}}\\
			\end{center}
			\caption{A junction with a cubic block of type I}\label{l=2}
		\end{figure}
		\item For $l=1$, the dessin $\Gamma$ is a junction of a cubic block $K$ and a block $B$ of degree 6. If $K$ is of type I, the same skeleton moves of $\Gamma$ as for $l=2$ carry an oval from $K$ to $B$, placing it near either end of the junction. The inverse moves and the move of Figure \ref{fig.Sk} (g) yield the skeleton indicated in  Figure \ref{alpha} ($l=1$). Therefore, the uniqueness of $\Gamma$ follows from its equivalence to the junction of a cubic block of type I and a block of degree 6 without ovals.
		\item For $l=0$, the dessin $\Gamma$ is a junction of a type II cubic block $K$   and a block $B$ of degree~6 without ovals and is therefore unique.
	\end{enumerate}
	If a curve $C\in\alpha^{+}_{lp}\<l\ra$ is constructed from a proper trigonal curve $C'\subset\Sigma_3 $, then $C'=N_3(C)$, and a curve in $\alpha^{-}_{lp}\<l\ra$, symmetric to $C$ with respect to $E_0$, can be obtained from a curve that in the affine chart $\Sigma_3\smallsetminus(E_3\cup N_3(F_p))$ is symmetric to $C'$ with respect to the $x$-axis, and obviously has the same dessin as $C'$.
		
	\end{proof}
	\begin{remark}
	Curves from the classes $\tilde{\alpha}\<l\ra$, $\alpha^{\pm}_{ov}\<l\ra$ differ from the curve $C\in\alpha^{\pm}_{lp}\<l\ra$ only by the choice of the point $a\in N_3(C)$.  Let $C'$ be the arc of the image of the branch $(1,0)$, bounded by singular points and having odd intersection multiplicity with the line $y=0$. The point $a$ lies on  $C'$  for  $C\in\tilde{\alpha}\<l\ra$, and on an oval  for  $C\in\alpha^{\pm}_{ov}\<l\ra$. Therefore, the same arguments prove the connectedness of these classes as well.
	
	\end{remark}
	\section{Application: Curves of genus 4 on a quadratic cone}\label{gen4}
	\begin{figure}[h]
	\begin{center}
		\scalebox{0.69}{\includegraphics{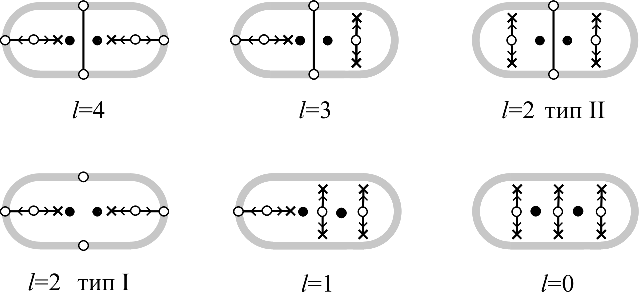}}\\
	\end{center}
	\caption{Skeletons of non-hyperbolic curves on $\Sigma_2$}\label{nhypcone}
	\end{figure}
	\begin{figure}[h]
	\begin{center}
		\scalebox{0.9}{\includegraphics{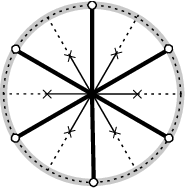}}\\
	\end{center}
	\caption{A dessin of a hyperbolic curve on $\Sigma_2$}\label{hypcone}
	\end{figure}
	The same arguments as in the proof of Lemma \ref{1white} yield a rigid isotopy classification of nonsingular real trigonal curves of genus 4 on the quadratic cone (cf. \cite[A3.6.1.]{DIK00}). After blowing up the vertex of the cone, such a curve yields a proper curve on $\Sigma_2$. The dessin of a nonhyperbolic curve is either a block of degree 6 or a junction of two cubic blocks (see Figure \ref{nhypcone}). Curves obtained from each other by reflection in the $x$-axis have the same dessin, but the real schemes symmetric to each other. A rigid isotopy class is determined by a complex scheme; their number is 11: for $l=0$, the scheme of a hyperbolic curve (see its dessin on Figure \ref{hypcone}, corresponding to an almost generic trigonal curve) and a scheme of type II; for $l=2$, a scheme of type I and two schemes of type II symmetric to each other; for $l=1, 3, 4$ -- two schemes symmetric to each other.
	
\begin{acknowledgments}
	The author thanks the reviewer for his advice and comments, which allowed  to correct errors and inaccuracies in the original version of the article.	
	
	The author's work was done on a subject of the State assignment FSWR-2023-0034.
\end{acknowledgments}	

\end{document}